\documentclass[10pt]{amsart}
\usepackage{amssymb,amsmath,amsthm,amsfonts,amsopn,url}
\usepackage{amscd,amssymb,amsopn,amsmath,amsthm,graphics,amsfonts,enumerate,verbatim,calc}
\usepackage[all]{xy}
\theoremstyle{plain}
\newtheorem{thm}{Theorem}[section]
\newtheorem*{mt*}{Theorem (Huneke-Lyubeznik)}
\newtheorem*{cj*}{Conjecture}
\newtheorem*{nt*}{Notations}

\newtheorem{lemma}[thm]{Lemma}
\newtheorem{cor}{Corollary}

\theoremstyle{definition}

\newcommand{\ideal}[1]{\mathfrak{#1}}

\newcommand{\m}{\ideal{m}}
\newcommand{\n}{\ideal{n}}

\newcommand{\func}[1]{\mathrm{#1} \,}

\newcommand{\height}{\func{ht}}

\newcommand{\Ass}{\func{Ass}}

\newcommand{\ZZ}{{\mathbb Z}}

\title[]{on local cohomology modules over ramified regular local rings}

\author[]{Rajsekhar Bhattacharyya}

\address{Dinabandhu Andrews College, Garia, Kolkata 700084, India}

\email{rbhattacharyya@gmail.com}

\thanks{}


\begin{document}

\begin{abstract}
In this paper, we show examples of local cohomology modules over ramified regular local ring, having finite set of associated primes. In doing so we consider our ramified regular local ring as Eisenstein extension of an unramified regular local ring sitting inside it. In ramified regular local ring for extended ideal (from the unramified one) set of associated primes of a local cohomology module is always finite. Using the Mayer-Vietoris spectral sequence, we construct examples of local cohomology of non extended ideal, whose set of associated primes is finite. In particular, we choose those non extended ideals whose minimal primes are extended ideals which are all cohomologically complete intersection. 
\end{abstract}

\maketitle
\section{introduction}

Let $R$ be a Noetherian ring and $M$ be a module over it. For an ideal $I\subset R$ and for some integer $i\geq 0$, consider local cohomology module $H^i_I (M)$ with support in the ideal $I$. In the fourth problem of \cite{Hu}, it is asked that whether local cohomology modules of Noetherian rings have finitely many associated prime ideals. There are examples given in \cite{Si}, \cite{Ka}, and \cite{SS}, which show that the set of associated primes of $i$-th local cohomology module $H^i_I (R)$ of a Noetherian ring $R$ can be infinite. However, there are several important cases where we have the finiteness of the set of associated primes of local cohomology modules and we list them below: (1) Regular rings of prime characteristic \cite{HS}, (2) regular local and affine rings of characteristic zero \cite{Ly1}, (3) unramified regular local rings of mixed characteristic \cite{Ly2} and (4) smooth algebra over $\ZZ$ \cite{BBLSZ}. On the other direction, the result is also true for every low dimensional regular local ring upto dimension five \cite{Boix}. These results support Lyubeznik conjecture, (\cite{Ly1}, Remark 3.7): 

\textbf{Conjecture:} Let $R$ be a regular ring and $I\subset R$ be its ideal, then for each $i\geq 0$, $i$-th local cohomology module $H^i_I (R)$ has finitely many associated prime ideals.

Recently, in (\cite{Ma}) it is shown that in ramified regular local rings, there exist examples of local cohomology modules with ifinitely many associated primes. To investigate the cause in more detail, in this paper, using the Mayer-Vietoris spectral sequence, we construct examples of local cohomology of non extended ideal, whose set of associated primes is finite. In particular, we choose those non extended ideals whose minimal primes are extended ideals which are all cohomologically complete intersection \cite{HeSc}. For detail, see, Theorem 2.3 and Corollary 1.
In (\cite{Nu3}), it is shown that if the set of associated primes of a local cohomology module does not contain the prime number $p$, then the set of associated primes is finite. 

The proof of the main theorem uses the properties of Eisenstein extension (see, for example \cite{Matsumura} for its reference) and the Mayer-Vietoris spectral sequence (see \cite{AGZ}, \cite{ABZ} and \cite{Ly4}). 

\section{the main result}

Let $(R;\m)$ be a local ring of mixed characteristic $p > 0$. We say $R$ is unramified if $p\notin {\m}^2$ and it is ramified if $p\in {\m}^2$. For normal local ring $(R,\m)$, consider the extension ring defined by $S=R[X]/f(X)$ where $f(X)=X^l+ a_1X^{l-1}+\ldots +a_l$ with $a_i\in \m$ for every $i=1,\ldots,l$ and $a_l\notin {\m}^2$. This ring $(S,\n)$ is local and it is defined as an Eisenstein extension of $R$ and $f(X)$ is known as an Eisenstein polynomial (see page 228-229 of \cite{Matsumura}). 

Here we note down the following important well known properties of Eisenstein extension:

(1) An Eisenstein extension of regular local ring is regular local (see Theorem 29.8 (i) of \cite{Matsumura}) and in this context, we observe that an Eisenstein extensions of an unramified regular local ring is a ramified regular local ring.

(2) Every ramified regular local ring is an Eisenstein extension of some unramified regular local ring (see Theorem 29.8 (ii) of \cite{Matsumura}. 

(3) The Eisenstein extension mentioned in (1) and (2) are faithfully flat (apply Theorem 23.1 of \cite{Matsumura}).

In proving our theorem, we need the Mayer-Vietoris spectral sequence, which arises from te Roos complex \cite{Roo} where in \cite{AGZ}, it is stated and proved only for defining ideal of an arrangement of linear subvarieties. In \cite{Ly4} it is extened for arbitrary commutative ring and module where \cite{ABZ}, the spectral sequnce has been constructed and properties are discussed very elaborately for more general categorical objects. Although the Mayer-Vietoris spectral sequences as observed in both \cite{AGZ} and \cite{Ly4}, differ in $E_0$ and $E_1$ page, but the coincide on $E_2$ page (see 4.2 Example \cite{ABZ} for detail).  

In this paper, we need the following theorem from \cite{Ly4} regarding the Mayer-Vietoris spectral sequence. 

\begin{thm}[Theorem on MVSS]\label{MV}
Let $A$ be an arbitrary commutative ring, let $I_1,\dots,I_n\subset A$ be ideals and
let $M$ be an $A$-module. There exists a spectral sequence
$$E_1^{-p,q}=\bigoplus_{i_0<\dots<i_p}H^q_{I_{i_0}+\dots+I_{i_p}}(M)
\Longrightarrow H^{q-p}_{I_1\cap\dots\cap I_n}(M).$$
\end{thm}

In our theorem we need the condition so that the above spectral sequence will be degenarate on $E_2$-page. In \cite{AGZ} in Theorem 1.2 it is shown the degenaracy in $E_2$-page for a defining ideal of an arrangement of linear subvarieties, when the underlying regular ring contains field. In \cite{ABZ} (see Lemma 4.15), degenaracy condition is analysed in a general setting and ideal which is cohomologically complete intersection (i.e. local cohomology module in the support of the ideal is non zero only at its height) used for the degenaracy. In \cite{Ly4} it has been mentioned that the spectral sequence will be degenarate on $E_2$ page if $A$ and $A/I_i$ are all regular (possibly arbitrary), with $M=A$ and gives reference of Theorem 1.2 of \cite{AGZ}. In our proof, we elaborate how the spectral sequence is degenate when $A$ and $A/I_i$ are regular with $M=A$. 

To prove the main theorem, we also need minimal primes of the support ideal to be cohomologicaly complete intersection whose sums are also so. For an ideal $I$ in regular local ring $S$, if $S/I$ is regular then $I$ is cohomologicaly complete intersections (although converse is not true, consider $k[[x,y,z]]$ and the ideal $(x^2 +y^3 + z^5)$), we can use them provided they satisfy the following lemma. 

\begin{lemma}
Let $(A,\o)$ regular local ring.\\ 
(a) An ideal $K$ will be generated by regular system of parameters, or equivalently $A/K$ is regular if and only if $K\cap {\o}^2\subset \o K$.\\
(b) If $I, J$ are ideals generated by part of regular system of parameters, or equivalently $A/I, A/J$ are regular rings and $A/(I+J)$ is also a regular then regular system of parameter of $I+J$ can get obtained by joining part of regular system of parameters of $I$ to that of $J$ and conversely if $I$ and $J$ are formed from a part of regular system of parameters then $I+J$ is generated by a regular system of parameters. 
\end{lemma}

\begin{proof}
(a) The asssertion will be true if and only if the minimal set of generators of $K$ is a regular system of parameters, since a regular system of parameters is a minimal set of generators for the ideal it generates. Now the minimala set of generators forms basis of the vector space $K/\o K$ over $A/\o$. So $K$ will be genetared by regular system of parameters if and only if the dimension of the vector space $(K+{\o}^2)/{\o}^2\cong K/K\cap {\o}^2$ is same to that of $K/\o K$ over $A/\o$. Since there is already a surjective canonical map from $K/\o K$ to $K/K\cap {\o}^2$, thus $K/\o K$ is canonically isomorphic to $K/K\cap {\o}^2$ if and only if $K\cap {\o}^2\subset \o K$. 

(b) Let $I=({u}_1,\ldots,{u}_k)A$ and $J=({v}_1,\ldots,{v}_l)A$ be generated by regular system of parameters and let $\{{u}_{i_1},\ldots,{u}_{i_a}, {v}_{j_1},\ldots,{v}_{j_b}\}$ be a set of minimal generatos of $I+J$. 
From (a) of this lemma it follows that they will form the regular system of parameters and our assertion follows. Converse follows from the fact that joining two part of same regular system of parameters will be a part of regular system of parameters.
\end{proof}

Now we prove the main result of the paper.

\begin{thm}
Let $(S,\n)$ be a 
complete ramified regular local ring of mixed characteristic which is an Eisenstein extension of an unramified complete regular local ring $R$ sitting inside it. Let $J$ be an ideal of $S$. Assume that minimal primes of $J$ are the extension of an ideals $I_t$ for $t=1,\ldots, n$ of $R$ such that the $S/\sum_{t\in F}I_t S$ are all regular local rings for every subsets $F\subset\{1,\dots,n\}$. Then $\Ass H^i_{J} (S)$ is finite.
\end{thm}

\begin{proof}

The basic idea of the proof is as follows: Since for every $F\subset\{1,\dots,n\}$, $S/\sum_{t\in F}I_t S$ are regular rings, the Mayer-Vietoris spectral sequence will degenerate on $E_2$-page. This will imply that $H^i_{J} (S)$ has a finite filtration such that every subquotient module of the filtration is isomorphic to a local cohomology module whose set of associated primes is finite. Thus the set of associated primes of $H^i_{J} (S)$ is finite. 

We apply the above theorem on MVSS for the regular local ring $S$ and ideal $J$ with $M=S$. Here we start from $E_1$ page and define $E_1^{-p,q}$ according to \cite{Ly4}, infact previously, it has actually defined in a more general context in \cite{AGZ}. Set $J_{i_0\ldots i_p}=I_{i_0}S+\dots+I_{i_p}S$ and $J_{i_0\ldots\widehat{i_j}\ldots i_p}=I_{i_0} S+\dots+I_{i_{j-1}} S+I_{i_{j+1}} S+\dots+I_{i_p} S$:
$$E_1^{-p,q}=\bigoplus_{i_0<\dots<i_p}H^q_{J_{i_0\ldots i_p}}(S)
\Longrightarrow H^{q-p}_{I_1 S\cap\dots\cap I_n S}(S)= H^{q-p}_{J}(S)$$

where differentials $E_1^{-p,q}\to E_1^{-p+1,q}$ on $E_1$-page are horizontal and they are defined according to \cite{Ly4} and in the following way: Consider their restriction to $H^q_{J_{i_0\ldots i_p}}(S)\subset E_1^{-p,q}$ where they take elements of it to some local cohomology module $H^{q}_{J_{i_0\ldots\widehat{i_j}\ldots,i_p}}(S)\subset E_2^{-p+1,q}$ and this map is induced by natural inclusion of $J_{i_0\ldots\widehat{i_j}\ldots i_p}\subset J_{i_0\ldots i_p}$.

Let $H^q_Q (R)$ be a nonzero direct summand of $E_1^{-p,q}$ and there is a map from here to $H^q_P (R)$ in $E_1^{-p+1,q}$ induced by the inclusion $P\subset Q$. Clearly $\height Q=q$. If $\height P=q$ then $Q=P$ and the map will be identity, otherwise for $\height P\neq q$, $H^q_P (R)=0$ and the map is a zero map. Now at each summand $H^q_{J_{i_0\ldots i_p}}(S)$ of $E_1^{-p,q}$ we map to $H^q_{J_{i_0\ldots\widehat{i_j}\ldots i_p}}(S)$ where $J_{i_0\ldots\widehat{i_j}\ldots i_p}\subset J_{i_0\ldots i_p}$. If for every $i_j$, $J_{i_0\ldots\widehat{i_j}\ldots i_p}\subset J_{i_0\ldots i_p}$ properly, then $H^q_{J_{i_0\ldots\widehat{i_j}\ldots i_p}}=0$ and the restriction of the differential on $H^q_{J_{i_0\ldots i_p}}(S)$ will be a zero map. Now the differential map which ends on $E_2^{-p,q}$, the image of it in $H^q_{J_{i_0\ldots i_p}}(S)$, is either the whole $H^q_{J_{i_0\ldots i_p}}(S)$ or zero. Since the modules on $E_2$-page are the homology modules of the differentials of $E_1$-page and since homology commutes with the direct sum, in this case either $H^q_{J_{i_0\ldots i_p}}(S)$ appears or will not appear in $E_2^{-p,q}$ as a direct summand. On the other hand, if for any one $i_j$, $J_{i_0\ldots\widehat{i_j}\ldots i_p}= J_{i_0\ldots i_p}$ then the map is identity, which implies kernel of this differential restricted at $H^q_{J_{i_0\ldots i_p}}(S)$ is zero and this particular local cohomology summand module can not appear in $E_2^{-p,q}$. Thus on the $E_2$ page, we get $E_2^{-p,q}\subset E_1^{-p,q}=\bigoplus_{i_0<\dots<i_p}H^q_{J_{i_0\ldots i_p}}(S)$ and the differentials are from $E_2^{-p,q}\to E_2^{-p+2,q-1}$ on the $E_2$-page. Consider its restriction to $H^q_{J_{i_0\ldots i_p}}(S)\subset E_2^{-p,q}$ where they take elements of it to $E_2^{-p+2,q-1}\subset E_1^{-p+2,q-1}= \bigoplus_{j_0<\dots<j_{p-2}}H^{q-1}_{J_{j_0\ldots j_{p-2}}}(S)$. If $H^{q-1}_{J_{j_0\ldots j_{p-2}}}(S)= 0$, then we have the zero map from $H^q_{J_{i_0\ldots i_p}}(S)$. Assume $H^{q-1}_{J_{j_0\ldots j_{p-2}}}(S)\neq 0$, this implies $\height J_{j_0\ldots j_{p-2}}= q-1$. Since $\height J_{i_0\ldots i_p}= q$, $J_{j_0\ldots j_{p-2}}$ can not contain $J_{i_0\ldots i_p}$ then according to Lemma 4.15 of \cite{ABZ}, there exists only the zero map from $H^q_{J_{i_0\ldots i_p}}(S)$ to $H^{q-1}_{J_{j_0\ldots j_{p-2}}}(S)$. So there is only zero differential in $E_2$-page.  
 
In fact for any positive integer $s$, inductively on $E_s$-th page, we have $E_{2}^{-p,q}= E_s^{-p,q}$ and by the same arguement as given to the $E_2$-page, the differential on the $E_s$-page is also zero. Thus $E_{2}^{-p,q}= E_{s+1}^{-p,q}$. Thus $E_{\infty}^{-p,q}= E_2^{-p,q}$. So for fixed $i=q-p$, $H^{i}_{J}(S)= H^{q-p}_{J}(S)= H^{q-p}_{I_1\cap\dots\cap I_n}(S)$ and we will get a filtration whose subquotient modules are isomorphic to $E_{\infty}^{-p,q}$ with varying $q$ and $p$ such that $i=q-p$. Since at each page of the spectral sequence we have only finitely many rows and columns, we have finite filtration and it is then sufficient to show that for each $(-p,q)$-pair, the set of asssociated primes of $E_{\infty}^{-p,q}$ is finite. From \cite{Mar} we know for every ideal $J_{j_0\ldots j_{p-2}}$ of $S$, $\Ass H^i_{J_{j_0\ldots j_{p-2}}}(S)$ is finite. So for each $(-p,q)$-pair $\Ass E_{\infty}^{-p,q}$ is finite. This implies $\Ass H^i_{J}(S)$ is contained in the subset of $\{J_{i_0\ldots i_p}:q-p=i\}$ and it is finite. This concludes the proof. 
\end{proof}

{\bf Remark 1}\\

We observe the following: For a ramified regular local ring $S$, let $x_1,\dots,x_d$ be a regular system of parameters such that $p, x_2,\dots,x_d$ is a system of parameters, then there exists a unramified regular local ring $R$, sitting inside it and $S$ is an Eisenstein extension of $R$. Any ideal generated by part of $x_2,\dots,x_d$ in $S$ is actually an extended ideal from that part of $R$. Moreover from (b) of Lemma 2.2 it is clear that it is necessary and sufficient to take all the ideals $I_t S$'s of ramified regular local ring from parts of one regular system of parameters so that they satisfy the condition of the above theorem.

So we can give a realization of the above theorem in the following situation.

\begin{cor}
Consider a part of regular system of parameters $x_2,\dots,x_k$ as told in Remark 2 and consider ideals $I_t S$ for $t=1,\ldots, n$ such that they are made out of parameters arbitrarily. Then for the ideal $J=\cap^n_{i=1}I_t S$ the set of associated primes of $H^i_{J}(S)$ is finite.
\end{cor}

\begin{proof}
Clearly from Lemma 2.2, $I_t S$'s satisfy the condition of the Theorem 2.3. Then we can use the result of Theorem 2.3 to the ideal $J=\cap^n_{i=1}I_t S$ to conclude.
\end{proof}

\section{acknowledgement}

I would like to thank Professor Alberto F. Boix for pointing out errors in the earlier manuscript and providing the counter example that every cohomologically complete intersection ideal is not such an ideal that its quotient ring is regular local.


\end{document}